\documentclass[11pt]{article}
\usepackage{amsmath, amsfonts, amssymb, amsthm}
\usepackage[left=1.2in, right=1.2in, top=1.1in, bottom=1.3in]{geometry}

\usepackage{pstricks,tikz}
	\usetikzlibrary{positioning,chains,fit,shapes,calc, fit}
	\usetikzlibrary{decorations.pathmorphing}

\usepackage{enumitem}
\usepackage{mathtools}

\newtheorem{theorem}{Theorem}[section]
\newtheorem{lemma}[theorem]{Lemma}
\newtheorem{claim}[theorem]{Claim}
\newtheorem{example}[theorem]{Example}
\newtheorem{conjecture}[theorem]{Conjecture}
\newtheorem{proposition}[theorem]{Proposition}
\newtheorem{observation}[theorem]{Observation}
\newtheorem{problem}[theorem]{Problem}
\newtheorem{definition}[theorem]{Definition}

\newenvironment{proofclaim}[1][Proof of Claim]{\begin{proof}[#1]}{\end{proof}}

\numberwithin{equation}{section}

\newcommand{\bbZ}{\mathbb{Z}^+}

\newcommand{\ep}{\varepsilon} 
\newcommand{\eps}{\ep} 
\newcommand{\tbf}[1]{\textbf{#1}}

\newcommand{\floor}[1]{\left\lfloor#1\right\rfloor}
\newcommand{\ceiling}[1]{\left\lceil#1\right\rceil}

\newcommand{\cA}{{\mathcal A}}

\def\al#1{}
	\renewcommand{\al}[1]{\footnote{\textbf{AL: }#1}}         
\def\ld#1{}
	\renewcommand{\ld}[1]{\footnote{\textbf{LD: }#1}}         
\def\refcom#1{}
	\renewcommand{\refcom}[1]{\footnote{\textbf{NB: }#1}}

\title{Spanning trees with few branch vertices} 
\author{Louis DeBiasio\thanks{Department of Mathematics, Miami University; {\tt  debiasld@miamioh.edu}. Research supported in part by Simons Foundation Collaboration Grant \# 283194.} \and Allan Lo\thanks{School of Mathematics, University of Birmingham; {\tt s.a.lo@bham.ac.uk}. The research leading to these results was partially supported by the EPSRC, grant no. EP/P002420/1.}}

\begin{document}
\maketitle

\begin{abstract}
A branch vertex in a tree is a vertex of degree at least three.
We prove that, for all $s\geq 1$, every connected graph on $n$ vertices with minimum degree at least $(\frac{1}{s+3}+o(1))n$ contains a spanning tree having at most $s$ branch vertices.  Asymptotically, this is best possible and solves, in less general form, a problem of Flandrin, Kaiser, Ku\u{z}el, Li and Ryj\'a\u{c}ek, which was originally motivated by an optimization problem in the design of optical networks.
\end{abstract}

\section{Introduction}

For a graph~$G$, the \emph{minimum degree of $G$}, denoted by $\delta(G)$, is the smallest degree of its vertices.
A \emph{tree} is an acyclic connected graph and a \emph{branch vertex} in a tree is a vertex of degree at least three.  Dirac~\cite{Dir} proved that every graph with minimum degree at least $(n-1)/2$ contains a Hamiltonian path, i.e.\ a spanning tree with no branch vertices and exactly two leaves; furthermore, this is best possible as for all $n\geq 2$, there are connected graphs with minimum degree $\ceiling{(n-1)/2}-1$ which have no Hamiltonian paths.  This result has been generalized in many ways.  In particular, Win \cite{Win} proved that if $G$ is a connected graph on $n$ vertices with $\delta(G)\geq (n-1)/k$, then $G$ contains a spanning tree in which every vertex has degree at most $k$.  Broersma and Tuinstra \cite{BT} proved that if $G$ is a connected graph on $n$ vertices with $\delta(G)\geq (n-k+1)/2$, then $G$ contains a spanning tree with at most $k$ leaves.  These results are best possible for all $k\geq 2$ and when $k=2$, they correspond to Dirac's theorem.  

The problem of determining whether a connected graph contains a spanning tree with a bounded number of branch vertices, while a natural theoretical question, seems to have been first explicitly studied because of a problem related to wavelength-division multiplexing (WDM) technology in optical networks, where one wants to minimize the number of \emph{light-splitting switches} in a \emph{light-tree} (see~\cite{GHHSV} for a more detailed description and background). 
Gargano, Hell, Stacho and Vaccaro~\cite{GHSV} showed that the problem of finding a spanning tree with the minimum number of branch vertices is NP-hard.
Since then, the problem has been investigated by many authors~\cite{CCGG,CCR,CGI,CS,LMS,Mar,MSU,RSS,SSRMGF,SLC,SSR}.

A spanning tree with at most one branch vertex is called a \emph{spider}.  Gargano, Hammar, Hell, Stacho and Vaccaro~\cite{GHHSV} (also see Gargano and Hammar~\cite{GH}) proved that if $G$ is a connected graph on $n$ vertices with $\delta(G)\geq (n-1)/3$, then $G$ contains a spanning spider (later Chen, Ferrara, Hu, Jacobson and Liu~\cite{CFHJL} proved the stronger result that connected graphs on $n\geq 56$ vertices with $\delta(G)\geq (n-2)/3$ contain a spanning \emph{broom}; that is, a spanning spider obtained by joining the center of a star to an endpoint of a path). Motivated by this, Gargano et al.~\cite{GHHSV} conjectured that for all $s\geq 1$, if $G$ is a connected graph on $n$ vertices with $\delta(G)\geq (n-1)/(s+2)$, then $G$ contains a spanning tree with at most $s$ branch vertices.  Later, Flandrin, Kaiser, Ku\v{z}el, Li and Ryj\'a\v{c}ek~\cite[Problem 11]{FKKLR} asked if the much stronger bound of $\delta(G)\geq n/(s+3)+C$ is sufficient and then Ozeki and Yamashita~\cite[Conjecture 30]{OY} conjectured a precise value for the constant term\footnote{In both places, the conjecture is stated as a generalized Ore-type degree condition; that is, in terms of the sum of the degrees of every independent set of $s+3$ vertices, but we only state the minimum degree version here.}.  

\begin{conjecture}[Ozeki and Yamashita~\cite{OY}]\label{mainconj}
For all $s\in \bbZ$, if $G$ is a connected graph on~$n$ vertices with $\delta(G)\geq \frac{n-s}{s+3}$, then $G$ contains a spanning tree with at most $s$ branch vertices.
\end{conjecture}

Note that even the approximate version of the conjecture by Flandrin et al. has not been verified for any $s\geq 1$ and the original (weaker) conjecture of Gargano et al. has not been verified for any $s\geq 2$.
The goal of this paper is to prove Conjecture \ref{mainconj} asymptotically.

\begin{theorem} \label{thm:main}
Let $s\in \bbZ$ and $\gamma >0$.
Then there exists $n_0 = n_0 (\gamma, s)$ such that every connected graph~$G$ on $n \ge n_0$ vertices with $\delta(G)\geq (\frac{1}{s+3}+\gamma) n $ contains a spanning tree with at most $s$ branch vertices.
\end{theorem}

The following example shows that our result is asymptotically best possible and that Conjecture~\ref{mainconj} is best possible if true.

\begin{example}\label{ex:lowerbound}
For all $s,m\in \mathbb{Z}^+$ with $m\geq 2$, there exists a connected graph $G$ on $n=(s+3)m-2$ vertices with $\delta(G)=\frac{n-s-1}{s+3}$ such that every spanning tree of $G$ has more than $s$ branch vertices.
\end{example}

\begin{proof}
Let $P = b_1 b_2 \dots b_{s+1}$ be a path on $s+1$ vertices and $H_1, H_2, \dots, H_{s+3}$  copies of complete graph on $m$ vertices such that $P, H_1, \dots, H_{s+3}$ are vertex-disjoint. 
For each $1 \le i \le s+1$, identify $b_i$ with a vertex of~$H_i$. 
We further identify $b_1$ and $b_{s+1}$ with a vertex of $H_{s+2}$ and $H_{s+3}$, respectively. 
We call the resulting graph~$G$, see Figure~\ref{lowerbound} for an example. 
Clearly $G$ has $n = (s+3)m-2$ vertices and $\delta(G) = m -1 = \frac{n-s-1}{s+3}$. 
Note that each $b_i$ is a branch vertex in any spanning tree of~$G$.  
\end{proof}

%

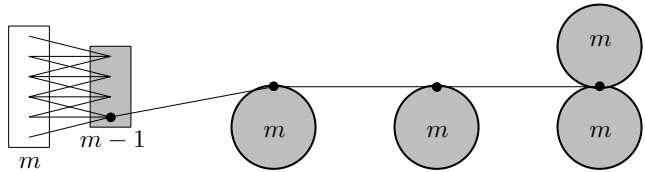
\begin{figure}[ht]
\centering
\begin{tikzpicture}[scale=1, line width = 1.5pt]
	\foreach \x in { 1,2,4,5}
		{
			\filldraw[fill=black!20, line width = 0.5 pt] ({2*\x},-0.75) circle [radius=0.75]; 
			\node at ({2*\x},-0.75)   {$K_m$};
		}
		\foreach \x in {1,3,5}
					{
						\filldraw[fill=black!20, line width = 0.5 pt] ({2*\x},0.75) circle [radius=0.75];
						\node at ({2*\x},0.75)   {$K_m$};
					}
		\draw (2,0) -- (10,0);
		
		\foreach \x in { 1,2,3,4,5 }
			{ 
				\node at ({2*\x},0)  (n\x) {};
				\filldraw[fill=black] (n\x) circle (1pt) ;
			}
		
		\node at (1.4,0)   {$b_1$};
		\node at (4,0.3)  {$b_2$};
		\node at (6,-0.3)   {$b_3$};
		\node at (8,0.3)  {$b_4$};
		\node at (10.6,0)  {$b_{5}$};
		\end{tikzpicture}
\caption[]{An example for the case $s=4$.}
\label{lowerbound}
\end{figure}

Our proof of Theorem~\ref{thm:main} uses the absorbing method, first systematically introduced by R\"odl, Ruci\'nski and Szemer\'edi~\cite{RRSz}, together with a non-standard use of Szemer\'edi's regularity lemma~\cite{Sz}.  In Section \ref{sec:partition} we discuss the canonical partition of the graph with linear minimum degree and then after stating Lemma~\ref{lma:component}, the main lemma of the paper, we use it to  deduce Theorem~\ref{thm:main}. In Section~\ref{sec:matching} we prepare for the proof of Lemma~\ref{lma:component} by proving a more basic result about (fractional) matchings. In Section~\ref{sec:regularity} we state the regularity lemma along with a few basic supporting lemmas. Finally, in Section~\ref{sec:cycle} we use the regularity lemma and the absorbing method together with the results of the previous section to prove Lemma~\ref{lma:component} which completes the result.  

\subsection{Notation}

For $n \in \mathbb{Z}^+$, we write $[n]$ for $\{1, \dots, n\}$.
We ignore floors and ceilings whenever they are not crucial to the calculation.  Throughout the paper, we will write $\alpha \ll \beta$ to mean that given $\beta$, we can choose $\alpha$ small enough so that $\alpha$ satisfies all of necessary conditions throughout the proof.  
In order to simplify the presentation, we will not determine these functions explicitly.

Let $G$ be a graph and $U,W \subseteq V(G)$ be disjoint.
We write $\overline{U}$ for $V(G)\setminus U$.
We denote by $G[U]$ and $G[U,W]$ the subgraph of~$G$ induced on~$U$ and the bipartite subgraph of $G$ induced by the partition $\{U,W\}$. 
We write $G \setminus U$ for $G[\overline {U}]$ and $e(U,W)$ for $e(G[U,W])$.
For $v \in V(G)$, $d_G(v,U)$ denotes the number of neighbors of~$v$ in~$U$.  
We write $N(U)$ for $\bigcup_{u \in U} N(u)$.
For graphs $G,H$, we write $G-H$ for the subgraph of~$G$ with vertex set $V(G-H)=V(G)$ and edge set $E(G-H) = E(G) \setminus E(H)$.

\section{Overview of the proof} \label{sec:sketch}

Our proof splits into two main parts.  First we show that if $G$ is a graph with minimum degree at least $(1/r+\gamma)n$, then we can find a  partition of $V(G)$ into at most $r-1$ parts $\{V_1, \dots, V_k\}$ having the property that for each $i$, $G[V_i]$ has no sparse cuts and most vertices in $V_i$ have degree at least $(1/r+\gamma/2)n$ in $G[V_i]$ while all other vertices in $V_i$ have linear minimum degree in $G[V_i]$.  Let us say that we have partitioned $G$ into ``robust'' subgraphs.

The second part of the proof focuses on these so-called robust subgraphs obtained above.  Let $t\geq 1$ and let $G$ be a graph on $n$ vertices with linear minimum degree having no sparse cuts in which most of the vertices have degree at least $(\frac{1}{t+3}+\gamma)n$.  We will show that not only does $G$ contain a spanning tree with at most $t$ branch vertices, but $G$ contains a cycle $C$ and a set $K\subseteq V(C)$ with $|K|\leq t$ such that for all $v\in V(G)\setminus V(C)$, $v$ has a neighbor in $K$.  It is clear that such a structure, which we call a ``star-cycle'', contains a spanning tree with at most $t$ branch vertices.  

The real heart of the proof lies in finding these spanning star-cycles in the robust subgraphs. 
By using the absorbing method in a particular form proved by the first author and Nelsen \cite{DN},
we can reduce the problem to finding a nearly spanning star-cycle. 
It is now standard in nearly spanning subgraph problems to use Sz\'emer\'edi's regularity lemma to reduce the problem to finding a simpler structure in the so-called reduced graph. 
For instance, if one were looking for a Hamiltonian cycle, it would be natural to apply the regularity lemma and prove that the reduced graph is connected and contains a perfect matching.
In our case, the simpler structure that we wish to find is a collection of vertex-disjoint edges and stars which we call a ``star-matching''.
Unfortunately it may not be sufficient to simply find a star-matching in the reduced graph, as this may not correspond to the desired star-cycle in the original graph.
Namely, there is no relationship between the maximum degree of the original graph and its reduced graph.  
For example, suppose that $G$ is the binomial random graph on~$n$ with each edge chosen independently with probability~$1/2$.
Then (typically) we have $\Delta (G) \approx n/2$ and the reduced graph~$R$ of $G$ is a complete graph on $k$ vertices for some large~$k$. 
Clearly, a spanning star~$S$ in~$R$ is a star-matching.
However, in order to `convert' $S$ into a nearly spanning star-cycle with one branch vertex in~$G$, we would need to find a star in $G$ with degree greater than $(1-1/k-o(1))n > n/2\approx \Delta(G)$. 
To get around this issue, we find the star-matching in the ``fractional-random-reduced-graph''~$R^*$ (see Definition~\ref{def:R*}) instead of the reduced graph, where $\Delta(R^*)$ `respects'~$\Delta(G)$.

Finally, to combine the two parts of the proof, we start with a connected graph $G$ having minimum degree at least $(\frac{1}{s+3}+\gamma)n$.  We obtain a robust partition of $G$ and inside each part of the partition we find a star-cycle having the correct number of stars depending on the relative degrees inside that part.  Then we use the connectivity of $G$ to find edges connecting the spanning star-cycles from each part of the partition.  The minimum degree of $G$ will put bounds on the number of parts of the partition and the relative degrees inside those parts in such a way that the obtained spanning tree has at most $s$ branch vertices.

\section{Sparse cuts and robust partitions}\label{sec:partition}

Let $0 < \alpha, \eta \le 1$ and $G$ a graph on $n$ vertices. 
For $X \subseteq V(G)$, we say that $(X, \overline{X})$ is an \emph{$\alpha$-sparse cut} if $e(X, \overline{X})<\alpha |X||\overline{X}|$.
Moreover, $G$ is \emph{$(\eta, \alpha)$-robust} if $\delta(G)\geq \eta n$ and $G$ has no $\alpha$-sparse cuts.  

We will use the following two simple observations from~\cite{DN}.

\begin{observation}[{\cite[Observation~4.4]{DN}}]\label{size}
Let $0<\alpha\leq \eta/2$, let $G$ be a graph on $n$ vertices, and let $\{X_1, X_2\}$ be a partition of $V(G)$ with $|X_1|\leq |X_2|$.  If $\delta(G) \ge \eta n$ and $|X_1| \le \eta n/2$, then $e(X_1, X_2) \ge \alpha |X_1| |X_2|$.
\end{observation}

\begin{observation}[{\cite[Observation~4.7]{DN}}]\label{slicerobust}
Let $0<\alpha\leq \eta/2$ and let $G$ be a graph on $n$ vertices.  If $G$ is $(\eta, \alpha)$-robust and $Z\subseteq V(G)$ with $|Z|\leq \alpha \eta n/8$, then $G \setminus Z$ is $(\eta/2, \alpha/2)$-robust.
\end{observation}


The following two lemmas are similar to Lemmas~6.1 and~6.2 in~\cite{CFS}; however, we cannot directly quote those results here as we need to use the fact that the relative degree of most vertices in each part of the partition is very close to their overall degree.

\begin{lemma} \label{prop:sparse}
Let $0 < \alpha < \eta \le \delta$ with $8 \alpha \le \eta$.
Let $G$ be a graph on $n$ vertices such that $\delta(G)\geq \eta n$ and $d(v) \ge \delta n$ for all but at most $\alpha n$ vertices $v\in V(G)$.
If $G$ has an $\alpha^2$-sparse cut, then there exists a partition~$\{Y_1,Y_2\}$ of~$V(G)$ such that, for all $i \in [2]$,
\begin{enumerate} 
	\item $ | Y_i | \ge  (\delta - 3 \alpha) n$;
	\item $\delta (G[Y_i]) \ge \delta(G) /2$, and $d(v,Y_i) \ge (\delta - 3 \alpha) n $ for all but at most $3  \alpha n$ vertices $v \in Y_i$.
\end{enumerate}
\end{lemma}

\begin{proof}
Let $\{X_1,X_2\}$ be a partition of~$G$ such that $e(X_1, X_2)<\alpha^2 |X_1||X_2|$ and $|X_1| \le |X_2|$.
By Observation~\ref{size}, $|X_1|, |X_2| \ge \eta n /2$.
Let $U_0$ be the set of vertices $v \in V(G)$ such that $d(v) < \delta n$, so $|U_0| \le \alpha n $.
For $i\in [2]$, let $U_i$ be the set of vertices $v \in X_i$ such that $d( v, X_{3-i}) \ge \alpha n $.
Since $e(X_1,X_2) \le \alpha^2 n^2$, we have $|U_i| < \alpha n $.
Let $X'_i : = X_i \setminus (U_i \cup U_0)$.
Note that $|X'_i| \ge \eta n /2 - 2 \alpha n \ge \eta n/4$ and for all $v \in X_i'$, 
\begin{align}
	d (v, X_i') & \ge d(v) - d( v, X_{3-i}) - |U_i| - |U_0| \ge d(v) - 3 \alpha n
	\nonumber \\
	& \ge \max\{ (\delta -3 \alpha) n, d(v) / 2 \}. \label{eqn:d(vX'_i)}
\end{align}
since $\delta \ge \eta \ge \eta/2 + 3 \alpha$ as $\alpha \le \eta/8$.
Partition $U_0\cup U_1\cup U_2$ into $U_1'$ and $U_2'$, such that $e(X'_1 \cup U'_1, X'_2 \cup U'_2)$ is minimized.
Let $Y_i = X'_i \cup U'_i$ for $i \in [2]$. 
Clearly $\{Y_1,Y_2\}$ is a partition of~$V(G)$.
Since $|U_0\cup U_1\cup U_2| \le 3 \alpha n$, $\{Y_1,Y_2\}$ satisfies (i) and (ii) by~\eqref{eqn:d(vX'_i)}.
\end{proof}

The next lemma shows that a graph $G$ can be partitioned into $\{V_1, \dots, V_k \}$ such that each $G[V_i]$ has no sparse cut and most of the vertices in $G[V_i]$ have very few neighbors outside of $V_i$.

\begin{lemma} \label{lma:partition}
Let $r,n \in \bbZ$ with $r\geq 2$ and $\gamma, \alpha  >0$ be such that $2^{2r+3}\alpha \le \min\{1/r, \gamma\} $. 
If $G$ is a graph on $n$ vertices with $\delta(G) \ge ( 1/r + \gamma ) n$, then there exists a partition $\{ V_1, \dots, V_{k}\}$ of $V(G)$ with $k \le r-1$ such that for each $i \in [k]$:
\begin{enumerate}
	\item $ | V_i | >  ( 1 / r + \gamma/2 ) n$;
	\item $\delta (G[V_i]) \ge \delta(G) / 2^{k-1}$, and $d(v,V_i) \ge ( 1 / r + \gamma/2) n $ for all but at most $4^{k+1} \alpha n$ vertices $v \in V_i$;
	\item $G[V_i]$ has no $16^{k+1} \alpha^2$-sparse cuts.
\end{enumerate}
\end{lemma}

\begin{proof}
Let $\mathcal{P}_1 := \{ V(G) \}$.
At step $j \le r-1$, suppose that we have already found a partition $\mathcal{P}_j = \{ U_1, \dots, U_j\}$ of $V(G)$ such that for all $i \in [j]$,
\begin{enumerate} [label ={\rm (\roman*$'$)}]
	\item $ | U_i | \ge ( 1 / r + \gamma - 4^{j+1} \alpha) n$;
	\item $\delta (G[U_i]) \ge \delta(G) / 2^{j-1}$, and  $d(v,U_i) \ge( 1 / r + \gamma - 4^{j+1} \alpha) n$ for all but at most $4^{j+1} \alpha n$ vertices $v \in U_i$.
\end{enumerate}
If $G[U_i]$ has no $16^{j+1} \alpha^2$-sparse cuts for all $i\in [j]$, then we are done by setting $k:=j$ and $V_i :=U_i$ for each $i \in [k]$.
So suppose without loss of generality that $G[U_1]$ has an $16^{j+1} \alpha^2$-sparse cut.
By Lemma~\ref{prop:sparse}, there is a partition $U'_1,U'_2$ of $U_1$ such that for $i \in [2]$
\begin{enumerate} [label ={\rm (\roman*$''$)}]
	\item $ | U'_i | \ge ( 1 / r + \gamma - 4^{j+1} \alpha) n  - 3 \cdot 4^{j+1} \alpha |U_1| \ge ( 1 / r + \gamma - 4^{j+2} \alpha  ) n $;
	\item $\delta (G[U'_i]) \ge \delta (G[U_1]) /2 \ge \delta(G) / 2^{j}$, and 
\[d(v,U'_i) \ge ( 1 / r + \gamma - 4^{j+1} \alpha) n  - 3 \cdot 4^{j+1} \alpha |U_1| \ge ( 1 / r + \gamma - 4^{j+2} \alpha  ) n\] for all but at most $3 \cdot 4^{j+1} \alpha  |U_1| \le 4^{j+2}  \alpha n$ vertices $v \in U'_i$.
\end{enumerate}
Set $\mathcal{P}_{j+1} := \{ U'_1,U'_2, U_2, \dots, U_j\}$ to be the partition of~$V(G)$ and note that (i$'$) implies that this process will end with a partition having at most $r-1$ parts.
\end{proof}

For $t \in \bbZ$, a \emph{$t$-star-cycle} is a union of cycle $C$ and $t$ vertex-disjoint stars $S_1, \dots, S_t$ such that the centers of stars are in~$V(C)$ and the leaves of the stars are not in $V(C)$.  The next lemma shows that each $G[V_i]$ obtained from Lemma~\ref{lma:partition} contains a spanning $t$-star-cycle for some $t$ depending on the relative degrees.  In fact, we show that when $G$ has no $\alpha$-sparse cuts, we can get an improvement in the bound on the degrees (note that $\frac{n}{s+3}\geq \frac{n}{(\sqrt{s}+1)^2}$ for all $s\geq 1$).

\begin{lemma} \label{lma:component}
Let $s,n\in \bbZ$ and let $ 1/n  \ll \alpha, \alpha' \ll \eta,\gamma,1/s$.  
If $G$ is an $(\eta, \alpha)$-robust graph on $n$ vertices such that $d(v) \ge ( \frac{1}{(\sqrt{s}+1)^2}  + \gamma ) n$ for all but at most $ \alpha' n $ vertices $v \in V(G)$, then $G$ has a spanning $t$-star-cycle with some $t \le s$. 
\end{lemma}

We will prove Lemma~\ref{lma:component} in Section \ref{sec:cycle}, but first we deduce Theorem~\ref{thm:main} using Lemma~\ref{lma:component}. 
We need the following well-known result of P\'osa.
\begin{theorem}[P\'osa~\cite{P}] \label{Posa}
Let $G$ be a graph on $n$ vertices. 
If for every $1\leq k \le (n-1)/2$, at most $k-1$ vertices have degree at most $k$, then $G$ contains a Hamiltonian cycle. 
\end{theorem}


\begin{proof}[Proof of Theorem~$\ref{thm:main}$]
Let $\alpha^*$ be a constant such that $ 1/n  \ll \alpha^* \ll \eta,\gamma,1/s$.
Let $\alpha : = (4^r\alpha^*)^2$ and $\alpha' : = 4^{s+3} \alpha^*$, so we have $\alpha' \le 2^{-(s+1)}(s+3)^{-2}$.
By Lemma~\ref{lma:partition} with $r=s+3$, there exists a partition $\{ V_1, \dots, V_{k}\}$ of $V(G)$ with $k\leq s+2$ such that for each $j \in [k]$,
\begin{enumerate}[label={\rm (\roman*)}]
	\item $ | V_j | >  ( \frac{1}{s+3} + \frac{\gamma}{2} ) n$,
	\item $\delta (G[V_j]) \ge \delta(G) / 2^{s+1}$ and for all but at most $\alpha' n $ vertices $v \in V_j$, $d(v,V_j) \ge (\frac{1}{s+3} + \frac{\gamma}{2}) n $, and
	\item $G[V_j]$ has no $\alpha$-sparse cut.
\end{enumerate}

For each $j \in [k]$, let $G_j:= G[V_j]$ and $s_j := \floor{\frac{(s+3)|V_j|}{n+1}}$.
Note that by (i), each $s_j \ge 1$.  Furthermore, 
\begin{equation}
\sum_{j \in [k]} s_j =\sum_{j \in [k]} \floor{\frac{(s+3)|V_j|}{n+1}}\leq \floor{\sum_{j \in [k]}\frac{(s+3)|V_j|}{n+1}}= \floor{\frac{n(s+3)}{n+1}}\leq s+2. \label{eqn:sums}
\end{equation}
Consider any $j \in [k]$.
Note that by the definition of $s_j$, we have $\frac{|V_j|(s+3)}{n+1}<s_j+1$ and thus 
\[
d_{G_j} (v) \ge \left( \frac{1}{s+3} + \frac{\gamma}{2} \right) n\ge  \left( \frac{1}{s_j+1} + \frac{\gamma}{4} \right) |V_j|
\]
for all but at most $\alpha' n\leq  (s+3) \alpha' |V_j| $ vertices $v \in V_j$. 
Also
\begin{align*}
	\delta(G_j) \ge  \delta(G) / 2^{s+1} \ge  \frac{n}{2^{s+1}(s+3)} >  (s+3) \alpha' |V_j|.  
\end{align*}
If $s_j = 1$, then Theorem~\ref{Posa} implies that $G_j$ has a Hamiltonian cycle $H_j$.
If $s_j \ge 2$, then Lemma~\ref{lma:component} implies that $G_j$ contains a spanning $t_j$-star-cycle $H_j$ with $t_j \le \max\{1, s_j-2\}$.
(Note that if $s_j\in\{2, 3\}$, then $t_j\leq 1$.) 
Therefore, each $G_j$ contains a spanning $t_j$-star-cycle with 
\begin{align}
t_j \le s_j-1. \label{eqn:si} 
\end{align}

Since $G$ is connected, there exist edges $e_1, \dots, e_{k-1}$ in $G$ such that $\bigcup_{j \in [k]} H_j \cup \bigcup_{j \in [k-1]} e_j$ is connected.
Without loss of generality, we may assume that for all $i \in [k-1]$, $e_i \cap V(H_{i+1}) \ne \emptyset$ and $e_i \cap \bigcup_{j \in [i]} V(H_{j}) \ne \emptyset$.
We claim that for each $i \in [k-1]$, there exists a tree $T_i$ spanning $\bigcup_{j \in [i+1]} H_{j} \cup \bigcup_{j \in [i]} e_j $ with at most $i-1 + \sum_{j \in [i+1]} t_j$ branch vertices.
We will proceed by induction on~$i$. 
For $i =1$ and $j \in [2]$, let $e'_j$ be an edge in the cycle of $H_j$ such that $e'_j$ intersects~$e_1$ if possible. 
Then $T_1 : = (H_1 -e'_1) \cup (H_2-e'_2) \cup e_1$ is a tree with $t_1+t_2$ branch vertices. 
Hence we may assume that $i >1$ and the statement holds for $i' < i$. 
Let $T_{i-1}$ be a spanning tree of $\bigcup_{j \in [i]} H_{j} \cup \bigcup_{j \in [i-1]} e_j $ with at most $i-2 + \sum_{ j \in [i] } t_j$ branch vertices (which exists by the induction hypothesis). 
Let $T'_{i+1}$ be a spanning tree of $H_{i+1} \cup e_{i}$ with exactly $t_{i+1}$ branch vertices.
(To be precise, $T'_{i+1} := H_{i+1} \cup e_{i} - e'_{i+1}$, where $e'_{i+1}$ is an edge in the cycle of $H_{i+1}$ such that $e'_{i+1}$ intersects~$e_i$ if possible.)
Thus $T_{i} := T_{i-1} \cup T'_{i+1}$ is a spanning tree of $\bigcup_{j \in [i+1]} H_{j} \cup \bigcup_{j \in [i]} e_j $ with at most $i-1 + \sum_{j \in [i+1]} t_j$ branch vertices, so the claim holds.

\begin{figure}[ht]
\centering
\includegraphics[scale=1]{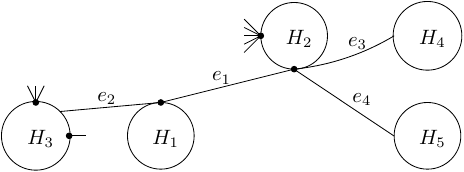}
\caption[]{Building the tree using the star-cycles $H_1, \dots, H_k$.  The branch vertices are highlighted in bold.}
\label{fig:partition}
\end{figure}

Let $T:= T_{k-1}$. Hence $T$ is a spanning tree of $G$ with at most 
\begin{align*}
k-2 + \sum_{j \in [k]} t_j = \sum_{j \in [k]} (t_j+1) -2 
\overset{\eqref{eqn:si}}{\le} \sum_{j \in [k]} s_j -2 
\overset{\eqref{eqn:sums}}{\le} s
\end{align*}
branch vertices, as desired.
\end{proof}

\section{Star-matchings}\label{sec:matching}

In this section we prove a preliminary result which will we will use together with the regularity lemma (see Lemma~\ref{lma:degreeform}) to prove Lemma~\ref{lma:component} in Section \ref{sec:cycle}.

For our purposes, we define a \emph{$2$-matching} to be a vertex-disjoint union of edges and odd cycles (sometimes this is referred to as a \emph{basic} 2-matching).
The \emph{order} of a $2$-matching $M$ is the number of vertices in $M$, and a \emph{maximum $2$-matching} is one of maximum order. 
Note that a $2$-matching~$M$ implies the existence of a (perfect) fractional matching on~$V(M)$. 
We need the following theorem of Pulleyblank which gives a Gallai-Edmonds-type (\cite{Ed},~\cite{Gal}) structural result for $2$-matchings.  Below we just state a simplified version of the result which suffices for our purposes, so the reader should see~\cite[Theorem 4]{Pul} for the complete statement.

\begin{theorem}[Pulleyblank~\cite{Pul}]\label{thm:2GalEd}
Let $G$ be a graph, let $A$ be the set of vertices which are not covered in some maximum matching, and let $A_1$ be the set of isolated vertices in $G[A]$.  If $M$ is a maximum $2$-matching for which the number of vertices contained in odd cycles is minimized, then $V(G) \setminus V(M) \subseteq A_1$ and the edges of $M$ incident with vertices in $A_1$ induce a matching which covers~$N(A_1)$.
\end{theorem}

For $t \in \mathbb{Z}^+$, a \emph{$t$-star-matching} is a vertex-disjoint collection of edges and exactly $t$ (non-trivial) stars and and a \emph{$t$-star-$2$-matching} is a vertex-disjoint collection of edges, odd cycles, and exactly $t$ (non-trivial) stars.  

\begin{lemma} \label{lma:starbipartite}
Let $n, s\in \bbZ$ and let $G$ be a bipartite graph on $n$ vertices with partition $\{A,B\}$.
If $d(a) \ge \frac{n}{(\sqrt{s}+1)^2}$ for all $a \in A$ and there exists a matching covering~$B$, then $G$ contains a spanning $t$-star-matching for some $t \le s$.
\end{lemma}

\begin{proof}
Let $M$ be a matching covering~$B$.  
We begin with two claims.
\begin{claim}\label{claim:M'}
Suppose there exists $B'\subseteq B$ with $|B'|\leq s$ and a matching $M'$ covering $A\setminus N(B')$. 
Then $G$ contains a spanning $t$-star-matching for some $t\leq s$.
\end{claim}

\begin{proofclaim}
Since $M$ and $M'$ cover~$B$ and $A\setminus N(B')$ respectively, there exists a matching $M^*$ which covers $B \cup (A\setminus N(B')) = V(G) \setminus N(B')$\footnote{See~\cite[Exercise 3.1.13]{West}}.
Finally, since $V(G) \setminus V(M^*) \subseteq N(B')$, there exists a spanning $t$-star-matching for some $t \leq |B'| \leq s$.  
\end{proofclaim}

\begin{claim}\label{claim:Hall}
If there exists $U\subseteq A$ such that $|N(U)|<|U|$, then $|U|>\frac{n}{(\sqrt{s}+1)^2}$ and there exists $b\in N(U)$ such that $|d(b, U)|\geq \frac{|U| n}{ (\sqrt{s}+1)^2 |N(U)| }>\frac{n}{(\sqrt{s}+1)^2}$.
\end{claim}

\begin{proofclaim}
Let $U\subseteq A$ such that $|N(U)|<|U|$.  Since $d(a) \ge \frac{n}{(\sqrt{s}+1)^2}$ for all $a \in A$, we clearly have $|U|>\frac{n}{(\sqrt{s}+1)^2}$.  Furthermore, by averaging, there exists a vertex $b\in N(U)$ such that  
$$d(b, U)\geq \frac{e(U, N(U))}{|N(U)|}\geq \frac{|U|\frac{n}{(\sqrt{s}+1)^2}}{|N(U)|}>\frac{n}{(\sqrt{s}+1)^2}.$$
\end{proofclaim}

Recall that $M$ covers $B$, so we may assume that $M$ does not cover $A$ or else we are done. 
Thus we have $|A| > |B|$.
By Claim~\ref{claim:Hall}, there exists $b_1\in B$ with $A_1:=N(b_1)$ such that 
\begin{align}\label{eq:A1}
|A_1|=|d(b_1, A)|\geq \frac{|A|\frac{n}{(\sqrt{s}+1)^2}}{|B|}\geq |A|-\frac{sn}{(\sqrt{s}+1)^2}.  
\end{align}
To see this last inequality, set $|A|=\alpha n$ (and so $|B|=(1-\alpha)n$), and then divide both sides by $n$.  Now it is straightforward to verify that $\frac{\alpha}{1-\alpha}\frac{1}{(\sqrt{s}+1)^2}\geq \alpha-\frac{s}{(\sqrt{s}+1)^2}$ holds for all $1/2\leq \alpha<1$ and $s\geq 1$.

Now suppose that for some $r \in [s-1]$ we have chosen vertices $b_1, \dots, b_{r}$ and pairwise disjoint sets $A_1, \dots, A_{r}$ such that $A_i = N(b_i)\setminus \bigcup_{j \in [i-1]} A_j$ for all $i \in [r]$, $|A_1|$ satisfies~\eqref{eq:A1}, and $|A_i|>\frac{n}{(\sqrt{s}+1)^2}$ for all $2\leq i\leq r$.  
If there exists a matching covering $A\setminus \bigcup_{i\in[r]}A_i$, then we are done by  Claim~\ref{claim:M'}  (with $B':=\{b_1, \dots, b_r\}$).
Otherwise, by Hall's theorem and Claim~\ref{claim:Hall}, there is a vertex $b_{r+1}$ with $A_{r+1}= N(b_{r+1})\setminus \bigcup_{i \in[r]}A_i$ such that $|A_{r+1}|> \frac{n}{(\sqrt{s}+1)^2}$. 
Thus we obtain vertices $b_1, \dots, b_{s}$ and pairwise disjoint sets $A_1, \dots, A_{s}$ such that $A_i= N(b_i)\setminus \bigcup_{j \in[i-1]}A_j$ for all $i \in [s]$, $|A_1|$ satisfies \eqref{eq:A1}, and $|A_i|>\frac{n}{(\sqrt{s}+1)^2}$ for all $2\leq i\leq s$.  Note that 
\begin{align*}
|A\setminus \bigcup_{i \in [s]}A_i| < |A|-\left(|A|-\frac{sn}{(\sqrt{s}+1)^2}\right)-\frac{(s-1)n}{(\sqrt{s}+1)^2}=\frac{n}{(\sqrt{s}+1)^2}.
\end{align*}
Since $d(a) \ge \frac{n}{(\sqrt{s}+1)^2}$ for all $a \in A$, there is a matching $M'$ saturating $A\setminus \bigcup_{i \in[s]}A_i$ and thus we are done by Claim~\ref{claim:M'} (with $B':=\{b_1, \dots, b_s\}$).
\end{proof}

We now combine Theorem~\ref{thm:2GalEd} and Lemma~\ref{lma:starbipartite} to prove the following result on spanning $t$-star-$2$-matchings in general graphs.  

\begin{lemma}\label{lma:star2match}
Let $n, s\in \bbZ$ and $0 <\alpha'\leq \eta, \gamma$.  If $G$ is a graph on $n$ vertices with $\delta(G)\geq \eta n$, and $d(v)\geq (\frac{1}{(\sqrt{s}+1)^2}+\gamma)n$ for all but at most $\alpha'n$ vertices $v\in V(G)$, then $G$ has a spanning $t$-star-$2$-matching with $t \le s$.
Moreover, if $G$ is bipartite, then $G$ has a spanning $t$-star-matching with $t \le s$. 
\end{lemma}

\begin{proof}
Let $A$ be the set of vertices in $G$ which are not covered in some maximum matching in~$G$ and let $A_1$ be the set of isolated vertices in $G[A]$.  
Let $M$ be a maximum $2$-matching in $G$ with the minimum number of vertices in odd cycles.  Let $B:=N(A_1)$ and let $H$ be the bipartite graph induced by $(A_1, B)$.  By Theorem~\ref{thm:2GalEd}, we have that $A_1$ is an independent set and the edges of $M$ in $H$, call them $\hat{M}$, induce a matching covering $B$.  
Let $A_1'$ be the set of at most $\alpha'n$ vertices $v\in A_1$ for which $d(v, B)< (\frac{1}{(\sqrt{s}+1)^2}+\gamma)n$, but $d(v, B)\geq \eta n$.  
Now by the size of $A_1'$, there exists a matching $M'$ covering~$A_1'$, which we will choose to have as many edges from $\hat{M}$ as possible.
Let $A_1^*=A_1\setminus V(M')$, $B^*=B\setminus V(M')$. 
Let $H^*= H[A_1^* \cup B^*]$ be the bipartite graph obtained from~$H$ by deleting vertices of~$M'$.
Note that $\hat{M}\setminus M'$ covers~$B^*$ and for all $v\in A_1^*$, 
\[
d(v, B^*)\geq \left(\frac{1}{(\sqrt{s}+1)^2}+\gamma \right)n-\alpha'n
\ge \frac{n}{(\sqrt{s}+1)^2} \ge \frac{|H^*|}{(\sqrt{s}+1)^2} .  
\]
Thus by Lemma~\ref{lma:starbipartite}, there is a spanning $t$-star-matching $M^*$ in $H^*$ with $t\le s$.   Now $(M\setminus \hat{M})  \cup M' \cup M^* $ gives us the desired $t$-star-$2$-matching of $G$.

If $G$ is bipartite, then since $G$ has no odd cycles, a $t$-star-$2$-matching is a $t$-star-matching.
\end{proof}

We note that any improvement in the bound on $d(a)$ for all $a \in A$ in Lemma~\ref{lma:starbipartite} would immediately improve the bounds in Lemma~\ref{lma:star2match}, Lemma~\ref{lma:spanning}, and consequently Lemma~\ref{lma:component}.

\begin{problem}
Determine the smallest value of $m$ so that the outcome of Lemma \ref{lma:starbipartite} holds with $d(a) \ge m$ for all $a \in A$.  It is at least $\frac{n}{2s+2}$ as witnessed by $s+1$ disjoint, nearly balanced copies of complete bipartite graphs on approximately $\frac{n}{s+1}$ vertices each.
\end{problem}

\section{Regularity lemma}\label{sec:regularity}

Let $G$ be a bipartite graph with bipartition~$\{A,B\}$. For non-empty sets $X\subseteq A$, $Y\subseteq B$, we define the \emph{density of $G[X,Y]$} to be $d_G(X, Y):=e_G(X,Y)/|X||Y|$.
Let $\ep>0$. We say that $G$ is \emph{$\ep$-regular} if for all sets $X \subseteq A$ and $Y \subseteq B$ with $|X|\geq \ep |A|$ and $|Y| \geq \ep |B|$ we have 
\begin{align*}
|d_G(A,B) - d_G(X,Y)| < \ep.
\end{align*}
The following simple results follow immediately from this definition.

\begin{proposition}\label{prop:typical}
Let $(A,B)$ be an $\eps$-regular pair with density $d$.
Then, for all $A'\subseteq A$ with $|A'|\geq \ep|A|$, all but at most  $2\ep|B|$ vertices in $B$ have $(d \pm \ep)|A'|$ neighbors in~$A'$.
\end{proposition}

\begin{proposition}\label{prop:slice}
Let $(A,B)$ be an $\eps$-regular pair with density $d$ and let $c > \eps$.
Let $A' \subseteq A$ and $B' \subseteq B$ with  $|A'| \ge c |A|$ and $|B'| \ge c |B|$.
Then $(A', B')$ is a $2\eps/c$-regular with density at least $d - \eps$.
\end{proposition}

Let $\mathcal{Q}$ be a partition of a set~$V$. 
For a subset $U \subseteq V$, define $\mathcal{Q} \setminus U : = \{ W \setminus U \colon W \in \mathcal{Q}\}$.
We say that a partition $\mathcal{Q}'$ is a \emph{refinement of $\mathcal{Q}$} if, for all $W' \in \mathcal{Q}'$, $W' \subseteq W$ for some $W \in \mathcal{Q}$.

Let $G$ be a graph on $V$. 
We say that $\mathcal{Q}$ is an $(\eps,d,m,k)$-regular partition of~$G$, if 
\begin{enumerate}[label={(Q\arabic*)}]
	\item $\mathcal{Q}= \{V_0, V_1,  \dots, V_k\}$ is a partition of $V$; 
	\item \label{itm:Q1} $|V_0| \le \eps n$;
	\item \label{itm:Q2} $|V_1| = \dots = |V_k| = m$;
	\item \label{itm:Q3} for all distinct $i,j \in [k]$, the graph $G[V_i, V_j]$ is $\eps$-regular and has density either $0$ or $>d$;
	\item \label{itm:Q4} for all $i \in [k]$, $G[V_i]$ is empty.
\end{enumerate}

We use the degree form of the regularity lemma.

\begin{lemma}[Degree form of the regularity lemma]\label{lma:degreeform}
For all $0<\eps<1$, there exists $N = N(\eps)$ such that the following holds for every $0\leq d < 1$.  For every graph $G$ on $n \ge N$ vertices with partition $\mathcal{Q'}$ of $V(G)$ into at most $\eps^{-1}$ parts, there exists a spanning subgraph~$G'$ of~$G$ and an $(\eps, d,m,k)$-regular partition $\mathcal{Q}= \{V_0, V_1,  \dots, V_k\}$ of $G'$ satisfying the following:
\begin{enumerate}
	\item $\eps^{-1}\leq k \leq N$;
	\item $\{V_1, \dots, V_k\}$ is a refinement of $\mathcal{Q}' \setminus V_0$;
	\item $\Delta ( G - G') \le  (d + \eps)n$.
\end{enumerate}
\end{lemma}

Lemma~\ref{lma:degreeform} can be derived from the original Szemer\'edi's regularity lemma~\cite{Sz}.
\footnote{See \cite[{Lemma~7.3}]{KOsurvey} for a sketch of the proof.}

We now define the reduced graph.

\begin{definition}[Reduced graph] \label{def:R}
Let $m,k \in \mathbb{Z}^+$ and $\eps, d>0$.
Let $G$ be a graph and $\mathcal{Q} = \{V_0, V_1, \dots, V_k\}$ an $(\eps, d,m,k)$-regular partition of~$G$. 
We define the \emph{$(\ep, d)$-reduced graph~$R$ of $G$} as follows. 
The vertex set of $R$ is the set of clusters $\{V_i : i \in [k]\}$.
For each $U,U' \in V(R)$, $U U'$ is an edge of $R$ if the subgraph $G[U,U']$ is $\eps$-regular and has density greater than $d$.
\end{definition}

Note that the $(\ep,d)$-reduced graph~$R$ depends on~$\mathcal{Q}$, which will always be known from the context. 
If $|V(G)| = n$ and $\delta(G) \ge \eta n$, then $\delta(R) \ge ( \eta - \eps ) k $ (see Proposition~\ref{prop:mindeg}).
As discussed in Section~\ref{sec:sketch}, there is no relationship between $\Delta(G)$ and $\Delta(R)$.
For our purpose, we would like that, if $d_{R} (V_i) = d_i k$, then the majority of the vertices $v \in V_i$ satisfies $d_{G} (v) \approx d_i n$.
We can achieve this by considering a graph $R^*$ obtained from the reduced graph $R$ by essentially replacing each edge $V_iV_j$ in $R$ of density $d_{i,j}$ with a random bipartite graph of density $d_{i,j}$.  Formally, 
we introduce the following notion.

\begin{definition}[Fractional-random-reduced graph] \label{def:R*}
Let $\ell, s \in \mathbb{Z}^+$ and $\eps, \eps' , d>0$.
Let $G$ be a graph and $\mathcal{Q} = \{V_0, V_1, \dots, V_k\}$ an $(\eps, d,m,k)$-regular partition of~$G$. 
For all distinct $i,j \in [k]$, let $d_{i,j}=d_{G}(V_i, V_j)$.
We say that a graph $R^*$ is an \emph{$(\ep', \ell,s)$-fractional-random-reduced graph of $G$} if 
\begin{enumerate}[label = {\rm (R\arabic*)}]
	\item \label{itm:R1}
	$V(R^*) = \bigcup_{i \in [k]} X_i$, where $X_i = \{ x_{i,j} : j \in [ \ell ]\}$ is a set of $\ell$ vertices;
	\item \label{itm:R2}
	for any distinct $i,j \in [k]$, if $d_{i,j} = 0 $, then $ R^* [ X_i , X_j ]$ is empty;
    \item \label{itm:R4} for any $i \in [k]$, any $s' \le s$ vertices $y_1, \dots, y_{s'}$ with each $y_p \in X_{i_p}$ where $i_p\neq i$, 
\begin{align*}	
	\left| X_i \cap \bigcup_{p \in [s']} N_{R^*} (y_p)  \right| 
	= 	\left| X_i \setminus \bigcap_{p \in [s']} \overline{N_{R^*} (y_p)}  \right| 
	=  (1 \pm \eps') \left( 1  - \prod_{p \in [s']} (1-d_{ i_p ,i})  \right) \ell.
\end{align*}	
\end{enumerate}
\end{definition}

Again, $R^*$ depends on $\mathcal{Q}$, which will be known from context.  
Note that if $0 < 1/\ell \ll \ep', d , 1/s$, then such an $R^*$ exists, by taking $R^*[X_i,X_j]$ to be a binomial random balanced bipartite graph on $2\ell$ vertices with probability $d_{i,j} $ for each distinct $i,j \in [k]$ (cf.~\cite[Definition 1.9]{KS}).

The key property of $R^*$ is that if $R^*$ contains vertex-disjoint stars covering some proportion of $V(R^*)$, then we can find vertex-disjoint stars in $G$ covering approximately the same proportion of $V(G)$ (see Lemma~\ref{lemma:star}).


The following simple propositions relate the minimum degrees of $G$, $R$, and $R^*$.

\begin{proposition}\label{prop:mindeg}
Let $n, \ell,s \in \mathbb{Z}^+$ and $ \eps, \eps', \delta, \eta, \alpha' >0$ with $\eps \le 1/2$.
Let $G$ be a graph on $n$ vertices with $\delta(G) \ge \eta n$. 
Let $U$ be the set of vertices $v \in V(G)$ such that $d(v) < \delta n $ and let $\mathcal{Q}' := \{ U, \overline{U}\}$.
Suppose that $ |U| \le \alpha' n $.
Let $\mathcal{Q} = \{V_0, V_1, \dots, V_k\}$ be an $(\eps, d,m,k)$-regular partition of~$G$ such that $\{V_1, \dots, V_k\}$ is a refinement of $\mathcal{Q}' \setminus V_0$.
Let $R$ and $R^*$ be the $(\ep,d)$-reduced graph and an $(\ep', \ell,s)$-fractional-random-reduced graph of $G$, respectively. 
Then 
\begin{enumerate}[label={\rm (\roman*)}]
	\item $\delta(R) \ge (\eta - \eps ) k$;
	\item $\delta (R^*) \ge ( \eta - \eps-\eps' )k\ell$;
	\item $d_{R^*}(x) \geq ( \delta - \eps -\eps' )k\ell$ for all but at most $2\alpha'k \ell$ vertices $x \in V(R^*)$.
\end{enumerate}
\end{proposition}

\begin{proof}
For all distinct $i,i' \in [k]$, let $d_{i,i'}=d_{G}(V_i, V_{i'})$.
Note that $k \le n/m \le 2k$.
For each $i \in [k]$, 
\begin{align}
	\sum_{i'\in [k]\setminus \{i\}} d_{i,i'}
	&\geq \frac{1}{m^2} \sum_{v\in V_i} ( d_{G}(v) - |V_0|)\notag\\
	&\geq \frac{k }{m}  \sum_{v\in V_i} \left( \frac{d_{G}(v)}n - \eps \right)
	\geq \frac{k }{m}  \sum_{v\in V_i} \left( \eta - \eps \right)
	= (\eta-\ep)k
	\label{eqn:dij}
\end{align}
implying~(i) as $d_R(V_i)  \ge \sum_{i'\in [k]\setminus \{i\}} d_{i,i'} $ and $\delta(G) \ge \eta n $.
Consider any vertex $x_{i,j}\in V(R^*)$. 
By \ref{itm:R4} (with $s'=1$) and~\eqref{eqn:dij}, we have 
\begin{align*}
	d_{R^*}( x_{i,j} ) 
	&= \sum_{i' \in [k] \setminus i} 	d_{R^*}( x_{i,j} , V_{i'})\\ 
	&\geq \sum_{i' \in [k] \setminus i} (1-\ep') d_{i,i'} \ell
	\ge (1-\ep')(\eta-\ep)k\ell
	\ge ( \eta - \eps-\eps' )k\ell.
\end{align*}
Hence $\delta (R^*) \ge ( \eta - \eps-\eps' )k\ell$ implying (ii). 
Recall that $\{V_1, \dots, V_k\}$ is an refinement of $\mathcal{Q}' \setminus V_0$. 
Thus for all but at most $\ell |U| / m  \le 2 \alpha' k \ell $ vertices $x \in V(R^*)$ satisfies $d_{R^*}(x) \geq ( \delta - \eps-\eps' )k\ell$.
\end{proof}

We need the following lemmas which give some desirable properties of $(\ep, d)$-reduced graph~$R$.
First, we show that if $G$ has no sparse cuts, then $R$ is connected.

\begin{lemma}\label{lma:connectedreduced}
Let $n,k,m \in \mathbb{Z}^+$ and $\ep, d, \alpha>0$ be such that $2 (d+2\eps) \le  \alpha $. 
Let $G$ be a graph on $n$ vertices with no $\alpha$-sparse cuts.
Let $G'$ be a spanning subgraph of $G$ with $\Delta(G - G') \le (d + \eps) n$.
Let $\mathcal{Q} = \{V_0, V_1, \dots, V_k\}$ be an $(\eps, d,m,k)$-regular partition of~$G'$.
Then the $(\ep, d)$-reduced graph~$R$ of~$G'$ is connected.
\end{lemma}

\begin{proof}
Suppose that $R$ is not connected.
Let $\cA$ be a component of $R$ with $|\cA| \le k/2$ and $A=\bigcup_{V_i\in \cA}V_i$.
Note that $|\overline{A}| \ge n/2$. 
By the hypothesis,
\begin{align*}
	e_G ( A, \overline{A}) \le  e_{G - G'} (A, \overline{A}) + e_G ( A, V_0) \le  (d +2\eps) n |A| \le \alpha |A| |\overline{A}|
\end{align*}
contradicting the fact that $G$ has no $\alpha$-sparse cuts.
\end{proof}

Next, we show that if $R$ is connected, then every pair of vertices in $V(G)\setminus V_0$ can be connected by a short path, even if some small number of vertices are forbidden to be used on the path. 
We say the \emph{length of a path~$P$} to mean its number of its edges. 
For vertices $u, v$, a \emph{$(u,v)$-path} is a path having $u$ and $v$ as endpoints. 

\begin{lemma}\label{lem:connect}
Let $n,k,m \in \mathbb{Z}^+$ and $\ep, d>0$ be such that $k \ge 3$ and $4 \eps < 3 d$. 
Let $G$ be a graph on $n$ vertices with $\delta(G) \ge \eta n $ and $\mathcal{Q} = \{V_0, V_1, \dots, V_k\}$ an $(\eps, d,m,k)$-regular partition of~$G$.
Suppose that $X \subseteq V(G)$ with $|X| \le d m / 4$ and the $(\ep, d)$-reduced graph~$R$ of~$G$ is connected.
Let $u,v \in V(G) \setminus X$ such that $d(u,V_i) \ge (d-\eps)m$ and $d(v,V_{i'}) \ge (d-\eps)$ for some $i,i' \in [k]$. 
Then there exists an $(u,v)$-path in $G \setminus X$ of length at most $k+1$.
\end{lemma}

\begin{proof}
By relabeling if necessary, suppose $d(u, V_1)\geq (d-\ep)m$ and $d(v, V_t)\geq (d-\ep)m$ such that $V_1V_2\dots V_t$ is a path in~$R$ (with $t\leq k$), which exists since $R$ is connected. 
If $t=1$, then let $V_2$ be a neighbor of~$V_1$ in~$R$. 
Note that by Proposition~\ref{prop:typical}, there exists a vertex in $V_2$ which has a neighbor in both $(N(u)\cap V_1)\setminus X$ and $(N(v)\cap V_1)\setminus X$.
Thus there exists an $(u,v)$-path of length~$4$ in~$G \setminus X$.

If $t\geq 2$, then set $V_1'=(N(u)\cap V_1)\setminus X$ and $V_{t}'=(N(v)\cap V_t)\setminus X$.
Apply Proposition~\ref{prop:typical} iteratively, we find $V'_2, \dots, V'_{t-1}$ such that for all $i \in [t-2]$, $V'_{i+1} \subseteq V_{i+1} \setminus X$, $|V_{i+1}'| \geq (d-\ep)|V_{i+1}|-|X| \geq \ep |V_{i+1}|$ and $d(v, V_i')\geq ( d-\ep )|V_{i}'|>0$ for all $v\in V_{i+1}'$.  
Note that $e_G(V_{t-1}', V_{t}') >0$ by Proposition~\ref{prop:typical}.
In the end we have a $(u,v)$-path $u x_1\dots x_t v$ with each $x_i \in V'_i$.
\end{proof}

The following Lemma appears explicitly in~\cite[Lemma 10]{BS}, although we only state a weaker version here.  It allows us to turn the existence of a matching in the reduced graph into long paths in the original graph.

\begin{lemma}\label{lem:edgetopath}
Let $0< m \ll  \ep < d/100$.
Let $G$ be a bipartite graph with bipartition $\{V_1, V_2\}$ with $|V_1|, |V_2|\geq m$.
Suppose that $G$ is $\ep$-regular with density at least $d/4$.
Then there exists a path of length at least $(2-10\ep/d)m$.
\end{lemma}

Let $R^*$ be a fractional-random-reduced graph of $G$. 
The next lemma shows that if $R^*$ contains vertex-disjoint stars covering some proportion of $V(R^*)$, then we can find vertex-disjoint stars in $G$ covering approximately the same proportion of $V(G)$. 

\begin{lemma}\label{lemma:star}
Let $n, \ell,s \in \mathbb{Z}^+$ and $ \eps, \eps',d>0$ with $\eps \le d$ and $ 4 s \eps \le \eps ' $ .
Let $G$ be a graph on $n$ vertices and $\mathcal{Q} = \{V_0, V_1, \dots, V_k\}$ an $(\eps, d,m,k)$-regular partition of~$G$.
Let $R^*$ be an $(\ep', \ell,s)$-fractional-random-reduced graph of~$G$. 
Suppose that $S_1, \dots, S_{t}$ are vertex-disjoint stars in~$R^*$ and $L$ is the set of leaves of $S_1, \dots, S_{t}$.
Then $G$ contains $t$ vertex-disjoint stars $S'_1, \dots, S'_t$ such that 
\begin{enumerate}[label = {\rm (\roman*)} ]
	\item $|\bigcup_{p \in [t] }V(S_p')| \ge  \frac{m}{\ell} | L | -4 \ep' m k $;
	\item for each $i \in [k]$, $| V_i \cap \bigcup_{p \in [t] }V(S_p') | \le \frac{m}{\ell} | X_i \cap L | +t$.
\end{enumerate}
\end{lemma}

\begin{proof}
For each $p \in [t]$, let $y_p$ be the center of $S_p$ with $y_p \in X_{i_p}$ and $L_p := V(S_p) \setminus y_p$ be the leaves of~$S_p$. 
For distinct $i,j \in [k]$, let $d_{i,j} := d_G(V_i,V_j)$. 

\begin{claim} \label{clm:star1}
There exist distinct vertices $v_1$, $\dots$, $v_{t} \in V(G)$ and subsets $I_1$, $\dots$, $I_t \subseteq [k]$ such that, for each $p \in [t]$ and all $i \in [k]$,
\begin{enumerate}[label = {\rm (\alph*$_p$)} ]
	\item \label{itm:star1}$v_p \in V_{i_p}$;
	\item \label{itm:star2} $\sum_{i \in I_p}  |L_p \cap X_i|  \ge (1- 4 \eps) |L_p| $;
	\item \label{itm:star3}if $i \in I_p$, then $L_p \cap X_i \ne \emptyset$;
	\item \label{itm:star4} $\displaystyle \left| V_i \setminus  \bigcup_{ p' \in [p] \colon i \in I_{p'} } N_G(v_{p'})  \right| 
		\leq m \max \left\{  \prod_{p' \in [p] \colon i \in I_{p'} } \left( 1 -  d_{i_{p'},i} + \eps \right), \ep \right\}$
\end{enumerate}
\end{claim}

\begin{proofclaim}
Suppose for some $p \in [t]$, we have already found  $v_1, \dots,v_{p-1}$, $I_1, \dots, I_{p-1}$.
We find $v_p$ and $I_p$ as follows. 

Consider $i \in [k] \setminus \{i_p\}$ with  $X_i \cap L_p \ne \emptyset$.
Let $J_i := \{ p' \in [p-1] \colon i \in I_{p'} \}$ and $N_i := \bigcup_{ p' \in J_i } N_G(v_{p'})$.
We say that the vertex $v \in V_{i_p}$ is \emph{$i$-good} if 
\begin{align*}
		\left| V_i \setminus \left( N_{G}(v) \cup N_i \right) \right| 
		\leq \max \left\{  \left(1- d_{i_p,i} + \eps \right) |N_i|, \ep m \right\},
\end{align*}
otherwise, we say that $v$ is \emph{$i$-bad}.
Define the function $\sigma_i : V_{i_p} \rightarrow \{0,1\}$ such that $\sigma_i(v) =1 $ if $v$ is $i$-good, and $\sigma_i(v) = 0$ otherwise.
Since $X_i \cap L_p \ne \emptyset$, then the definition of $R^*$ implies that $d_{i_p,i} \ge d>0$.  
If $\left| V_i \setminus N_i \right| \le \eps m$, then $\sigma_i(v) = 1$ for all $v \in V_{i_p}$.
Otherwise since $G[V_{i_p}, V_i]$ is $\eps$-regular, Proposition~\ref{prop:typical} implies that for all but at most $ 2\eps m$ vertices $v \in V_{i_p}$,
	\begin{align*}
		\left| V_i \setminus \left( N_{G'}(v) \cup N_i \right) \right| 
		& \leq \left(1-  d_{i_p,i} + \eps \right) \left| V_i \setminus N_i \right|.
	\end{align*}
Hence $\sigma_{i} (v) = 1$ for all but at most $ 2 \eps m$ vertices $v \in V_{i_p}$, that is, $\sum_{v\in V_{i_p}}\sigma_i(v)\geq (1-2\ep)m$.
Therefore,
\begin{align*}
	\sum_{v \in V_{i_p}}  \sum_{i \in [k] \colon X_i \cap L_p \ne \emptyset}  |X_i \cap L_p| \sigma_{i} (v) 
	\ge  (1- 2  \eps) m \sum_{i \in [k]} |X_i \cap L_p|
	= (1- 2 \eps)  m |L_p|.
\end{align*}
So, by averaging, there exists a vertex $v_{p} \in V_{i_p} \setminus \{v_1, \dots, v_{p-1}\}$ such that 
\begin{align*}
	 \sum_{i \in [k]}  |X_i \cap L_p| \sigma_{i} (v_p)  \ge (1- 4 \eps) |L_p| .
\end{align*}
Set $I_p := \{ i \in [k] \colon X_i \cap L_p \ne \emptyset$ and $\sigma_i (v_p) =1\}$.
Clearly \ref{itm:star1}--\ref{itm:star3} hold.
We now verify \ref{itm:star4}.
If $i \notin I_p$, then \ref{itm:star4} holds by~(d$_{p-1}$). 
If $i \in I_p$, then \ref{itm:star4} holds by (d$_{p-1}$) and the fact the $v_p$ is $i$-good.
\end{proofclaim}

We are going to construct vertex-disjoint stars $S'_1, \dots,S'_t$, where $S_p$ has center $v_p$ and leaves in $\bigcup_{i \in I_p} V_i$. 
Let $\hat{V} := \{v_1, \dots, v_t\}$.
Suppose that for each $i \in [k]$, 
\begin{align}
	\label{eqn:star1}
	|V_i \cap N_G(\hat{V}) | \ge \frac{m}{\ell} \left( \sum_{ p\in [t] \colon i \in I_p} |L_p \cap X_i| \right) - 2 \ep' m.
\end{align}
Pick $W_i \subseteq V_i \cap N_G(\hat{V})$ such that $	|W_i| = \frac{m}{\ell} \left( \sum_{ p\in [t] \colon i \in I_p} |L_p \cap X_i| \right) - 2 \ep' m$.
Together with~\ref{itm:star2} and the fact that $4s\eps \le \eps'$, we have
\begin{align*}
		\sum_{i\in [k]}|W_i|
		&\geq \sum_{i \in [k]}\left(  \frac{m}{\ell} \left( \sum_{ p\in [t] \colon i \in I_p} |L_p \cap X_i| \right) - 2 \ep' m \right)\\
		&= \frac{m}{\ell} \left(\sum_{p\in [t]} \sum_{i \in I_p}  |L_p \cap X_i|  \right) - 2  \ep' mk \\
		&\geq \frac{m}{\ell}\sum_{p\in [t]} (1- 4 \eps) |L_p| -2 \ep' mk
		\geq \frac{m}{\ell} |L|- (2\ep'+ 4 s\eps) m k
\geq \frac{m}{\ell} |L|- 3\ep' m k.
\end{align*}
As $W_i \subseteq V_i \cap \bigcup_{p \in [t]} N_{G}(\hat{V})$, $G$ contains vertex-disjoint stars $S'_1, \dots, S'_{t}$ with centers $v_1, \dots, v_t$ such that $\bigcup_{p \in [t]} V(S'_p) = \hat{V} \cup \bigcup_{i \in [k]} W_i$.
Clearly (i) holds. 
Since $|W_i| \le \frac{m}{\ell} |L \cap X_i|$ and $|\hat{V}| =t$, (ii) holds.

To see \eqref{eqn:star1} holds, consider $i \in [k]$. 
Let $J_i = \{p \in [t] \colon i \in I_p\}$. 
If $|V_i \cap N_G(\hat{V})| \ge (1- \eps)m$, then we are done as $\sum_{ p\in J_i} |L_p \cap X_i| \le |X_i| =  \ell$ and $\eps < \eps'$. 
Hence \ref{itm:star4} implies that 
\begin{align*}
	&\left|V_i \cap N_G(\hat{V}) \right|
	 \ge 
	\left| V_i \cap  \bigcup_{ p\in J_i } N_G( v_{p})  \right| 
	\ge
	\left( 1 -   \prod_{ p\in J_i } \left( 1 -  d_{i_{p},i} + \eps  \right) \right) m  
	\\ &
	\ge
	\left( 1 -  \prod_{p\in J_i} \left( 1 -  d_{i_{p},i}  \right) \right) m  - s \eps m 
	\overset{\mathclap{ \ref{itm:R4} }}{\ge} 
	\frac{m}{\ell} \left| X_i \cap \bigcup_{p \in J_i} N_{R^*}(y_p)  \right| - ( s \eps+  \eps') m\\
	& \ge \frac{m}{\ell} \left| X_i \cap \bigcup_{p \in J_i} N_{S_p}(y_p)  \right|  -  2 \eps'  m 
 =  \frac{m}{\ell}  \sum_{p\in J_i} |X_i \cap L_p|   - 2 \ep' m.
\end{align*}
Thus \eqref{eqn:star1} holds as required.
This completes the proof of the lemma. 
\end{proof}


\section{Spanning star-cycles in graphs with no sparse cuts}\label{sec:cycle}


Let $\beta>0$ and let $G$ be a graph on $n$ vertices.  We say $G$ is \emph{$\beta$-near-bipartite} if there exists $X\subseteq V(G)$ such that $e(X)<\beta n^2$ and $e(V(G)\setminus X)<\beta n^2$.

The proof of Lemma~\ref{lma:component} will be obtained by a combination of the following two results, the first of which is proved by the first author and~Nelsen in~\cite{DN}.  Lemma~\ref{lemma:absorbing} provides the existence of an absorbing path which depends on whether $G$ is near-bipartite or not.  In order to use the absorbing path in the case that $G$ is near-bipartite, we show in Lemma~\ref{lma:spanning} that the nearly spanning $t$-star-cycle can be chosen so that there are an equal number of leftover vertices in each part of the bipartition.

\begin{lemma}[Absorbing Lemma~\cite{DN}]\label{lemma:absorbing}
Let $0< 1/n\ll \alpha\ll \eta$, set $\rho:=\alpha^{32/\alpha^2}$, and suppose $G$ is an $(\eta, \alpha)$-robust graph on $n$ vertices.  
\begin{enumerate}
\item If $G$ is not $\alpha^4$-near-bipartite, then there exists a path $P$ with $|V(P)| \le \rho n $ such that for all $W\subseteq V(G)\setminus V(P)$ with $|W|\leq \rho^3 n$, the subgraph $G[V(P)\cup W]$ contains a spanning path having the same endpoints as $P$.  

\item If $G$ is $\alpha^4$-near-bipartite, then there exists a partition $\{A,B\}$ of $V(G)$ and a path~$P$ with $|V(P)| \le \rho n $ such that $\delta(G[A,B]) \ge \eta n /2$ and for all $W\subseteq V(G)\setminus V(P)$ with $|W\cap A|=|W\cap B|\leq \rho^3 n$, the subgraph $G[V(P)\cup W]$ contains a spanning path having the same endpoints as $P$.
\end{enumerate}
\end{lemma}

\begin{lemma}\label{lma:spanning}
Let $s,n \in \bbZ$ and be $\alpha, \alpha', \eta, \rho, \gamma$ be such that $	1/n \ll \rho , \alpha,\alpha' \ll \eta , \gamma  , 1/s$ and $\rho \le \alpha \eta /16s$.
Let $G$ be an $(\eta,\alpha)$-robust graph on $n$ vertices such that $d(v)\geq (\frac{1}{(\sqrt{s}+1)^2)}+\gamma)n$ for all but at most $\alpha'n$ vertices. 
Let $P\subseteq G$ be a path of order $|V(P)| \le \rho n$.
Then $G$ contains a $t$-star cycle~$C^*$ for some $t\leq s$ with $|V(C^*)| \ge (1- \rho^3)n$, which contains $P$ as a segment.  
Moreover, if there exists a partition $\{A,B\}$ of~$V(G)$ such that $\delta(G[A,B]) \ge \eta n/2$, then we can choose $C^*$ as above having the additional property that $|A\setminus V(C^*)|=|B\setminus V(C^*)|$.
\end{lemma}

\begin	{proof}
Define $N, \ell  \in \bbZ$ and let $d, \eps, \eps' $ be such that 
\begin{align*}
	1/n \ll 1/N \ll \eps \ll d, 1/\ell, \eps'\ll \rho, \alpha,\alpha' \ll \eta , \gamma  , 1/s.
\end{align*}
If $\{A,B\}$ is the partition of~$V(G)$ such that $\delta(G[A,B]) \ge \eta n/2$, then by Chernoff's bound for a random variable with hypergeometric distribution, there exists a set $W \subseteq V(G) \setminus V(P)$ such that $|W \cap A| = |W \cap B| = \rho^3 n/4$ and for all $v\in V(G)$, 
\begin{align}
	\label{eqn:res1}
	d_G ( v, A \cap W ) \ge \frac{\rho^3}{4} d_{G \setminus V(P)} (v ,A) - \eps n~\text{ and } d_G ( v, B \cap W ) \ge \frac{\rho^3}{4} d_{G \setminus V(P)} (v ,B) - \eps n.
\end{align}
We reserve $W$ to ensure $|A\setminus V(C^*)|=|B\setminus V(C^*)|$ later.
If such $\{A,B\}$ does not exist, then let $ W $ be a subset of $V(G) \setminus V(P)$ of size $|W| = \rho^3 n /2$.

Let $\hat{G}:=G \setminus (V(P) \cup W)$. 
Let  
\begin{align*}
V':= \left\{ v \in V(\hat{G}) \colon d_{\hat{G}}(v) < \left( \frac{1}{(\sqrt{s}+1)^2}+\frac{3\gamma}{4}\right)n \right\},
\end{align*}
so $ |V'| \le \alpha' n $.
Let $\mathcal{Q}' := \{ V', V(\hat{G}) \setminus V'\}$.
Apply Lemma~\ref{lma:degreeform} to $\hat{G}$ and $\mathcal{Q}'$ to obtain a spanning subgraph~$G'$ of~$\hat{G}$ and an $(\eps, d,m,k)$-regular partition $\mathcal{Q}= \{V_0, V_1,  \dots, V_k\}$ of $G'$ such that 
\begin{enumerate}[label = {\rm(a$_{\arabic*}$)}]
	\item \label{itm:a1}$\eps^{-1}\leq k \leq N$;
	\item $\{V_1, \dots, V_k\}$ is a refinement of $\mathcal{Q}' \setminus V_0$;
	\item \label{itm:a3}$\Delta ( \hat{G} - G') \le  (d + \eps)n$.
\end{enumerate}
Let $R$ be the $(\ep, d)$-reduced graph of~$G'$.
Note that since $\hat{G}$ is $(\eta/2, \alpha/2)$-robust by Observation~\ref{slicerobust} and the fact that $|V(P) \cup W| \le 2\rho n \le \alpha \eta n/8$, $R$ is connected by Lemma~\ref{lma:connectedreduced}.

We now reserve a set $U$, which will be used to connect paths into a cycle later, as follows. 
By Chernoff's bound for a random variable with hypergeometric distribution, there exists a set $U \subseteq \bigcup_{i \in [k]} V_i$ such that $|U \cap V_i| = \rho^8 m$ and, for all $v\in V(G)$ we have  
\begin{align}
	\label{eqn:res2}
	d_G ( v, U ) \ge \rho^8( d_G(v) - (\rho + \rho^3/2)n) -\ep n \ge  \rho^8 \eta n / 2\ge  \eta |U|/2,
\end{align}
where we use the facts that $\delta(G)\geq \eta n$, $|U| \le \rho^8 n$ and $\eps, \rho \ll \eta$.
Let $\mathcal{Q}_U = \{ \emptyset, V_1 \cap U, V_2 \cap U, \dots, V_k \cap U \}$.
Since $\eps \ll \rho, d$, Proposition~\ref{prop:slice} implies that $\mathcal{Q}_U$ is an $( \eps^{1/2} , d/2 , \rho^8 m,k)$-regular partition of~$G'[U]$.
Note that $R$ is also isomorphic to the $(\ep^{1/2}, d/2)$-reduced graph of~$G'[U]$.
So the $(\ep^{1/2}, d/2)$-reduced graph of~$G'[U]$ is connected.

Let $G^* = G' \setminus U$, $\mathcal{Q}^* = \mathcal{Q} \setminus U$ and $m^* = (1- \rho^8)m$.
Let $V^*_i = V_i \setminus U$ for all $i \in [k]$. 
Note that 
\begin{align}
	\label{eqn:m*k}
	m^*k  = (1- \rho^8)mk \ge (1- \rho^8) (1- \eps) |V(\hat{G})| \ge (1- \rho^8 - \eps - \rho^3/2) n - |V(P)|.
\end{align}
By Proposition~\ref{prop:slice}, $\mathcal{Q}^*$ is an $( 2 \eps , d/2 ,  m^*,k)$-regular partition of~$G^*$.
Let $R^*$ be an $(\ep', \ell,s)$-fractional-random-reduced graph of~$G^*$.
Note that  for all $v \in V(G^*)$, 
\begin{align*}
	d_{G^*} (v) 
	&\ge d_G(v) - |V(P)| -|W| - \Delta( \hat{G} - G') - |U|\\
	& \overset{\mathclap{\text{\ref{itm:a3}}}}{\ge} d_G(v) - (\rho + \rho^3/2 + d + \eps + \rho^8) n
	\ge d_G(v) - \eta |G^*|/3. 
\end{align*}
Proposition~\ref{prop:mindeg} implies that $\delta (R^*) \ge \eta k \ell / 2 $ and $d_{R^*}(x)  \ge \left( \frac{1}{(\sqrt{s}+1)^2} + \frac{\gamma}{2}\right) k \ell$ for all but at most $8 \alpha' k \ell$ vertices $x \in V(R^*)$.

By Lemma~\ref{lma:star2match}, $R^*$ has a spanning $t$-star-$2$-matching~$M^*$ with some $t \le s$. 
Let $S_1,\dots, S_{t}$ be the non-trivial stars of $M^*$ and let $M$ be the $2$-matching of~$M^*$.
Let $L$ be the set of leaves of $S_1, \dots, S_{t}$.
Since $t \le \ell$, \ref{itm:a1} implies that 
\begin{align}
\label{eqn:LM}
	|L| + |V(M)| = k \ell - t \ge (1- \eps) k\ell ~ \text{and}~| X_i \cap (L \cup V(M))| \le \ell.
\end{align}
By Lemma~\ref{lemma:star}, $G^*$ contains $t$ vertex-disjoint stars $S^*_1, \dots, S^*_t$ such that 
\begin{enumerate}[label = {\rm(b$_{\arabic*}$)}]
	\item \label{itm:b1}
	$|\bigcup_{p \in [t] }V(S^*_p) | \ge  \frac{m^*}{\ell} | L | -4 \ep' m^* k $;
	\item \label{itm:b2}
	for each $i \in [k]$, $| V^*_i \cap \bigcup_{p \in [t] }V(S^*_p) | \le \frac{m^*}{\ell} | X_i \cap L | +t$.
\end{enumerate}

Recall that $M$ is a $2$-matching.
Let $M=M_1\cup M_2$, where $M_1$ is the set of components of $M$ consisting of a single edge and $M_2$ is the set of components of $M$ which are odd cycles.
Let $H$ be the multigraph on $[k]$ such that $ij$ is an edge of $H$ of multiplicity $2\mu_1+\mu_2$, where $\mu_r$ is the number of edges between $X_i$ and $X_j$ in $M_r$.
Note that 
\begin{equation}\label{eqn:HM}
|E(H)| = |V(M)|.
\end{equation}
For each edge $e \in E(H)$, we choose $W_e \subseteq \bigcup_{i \in e} V_i \setminus \bigcup_{p \in [t] }V(S^*_p)$ such that $|W_e \cap V^*_i| = m^*/2\ell - t $ for each $i \in e$.
By~\eqref{eqn:LM} and~\ref{itm:b2}, we can ensure that $\{W_e \colon e \in E(H)\}$ is pairwise disjoint. 

Consider any $e = ii'\in E(H)$. 
Note that $G^*[W_e] = G^*[W_e \cap V^*_i, W_e \cap V^*_{i'}]$.
By Proposition~\ref{prop:slice}, $G^*[W_e]$ is $5 \ell \eps$-regular with density at least $d/4$.
Apply Lemma~\ref{lem:edgetopath} and obtain a path~$P_e$ in~$G^*[W_e]$ with 
\begin{align*}
|V(P_e)| \ge (1-100 \ep / d) m^*/\ell - 2t \ge ( 1 - 2\eps' ) m^*/\ell,
\end{align*}
where the last inequality holds as $\eps \ll d \ll \eps'$.
Recall that $|E(H)| = |V(M)| \le k \ell$. 
Note that
\begin{align}
\nonumber  & |V(P)| +|\bigcup_{p \in [t]}V(S^*_p)| + \sum_{e \in E(H) } |V(P_e)|\\
 \nonumber  & \overset{\mathclap{\text{\ref{itm:b1}}}}{\ge } |V(P)| + \frac{m^*}{\ell} | L | -4 \ep' m^* k  + ( 1 - 2\eps' ) \frac{m^*}{\ell}|E(H)|\\
			\nonumber  & \overset{\mathclap{\eqref{eqn:HM}}}{ \ge } |V(P)| + (1-2 \eps')(|L|+|V(M)|)\frac{m^*}{\ell} -4 \ep' m^* k \\
			\nonumber & \overset{\mathclap{\eqref{eqn:LM}}}{ \ge } |V(P)| + (1- 6\eps'- \eps)km^*\\
			& \overset{\mathclap{\eqref{eqn:m*k}}}{ \ge } (1- 6 \eps'- 2\eps - \rho^8- \rho^3/2)n
			> (1 - \rho^3/2-  2\rho^8)n, \label{eqn:V(C*)}
\end{align}
where for the last inequality we use the fact that $\eps, \eps' \ll \rho$.

Next, we connect $P$, $S_1^*$, $\dots$, $S^*_{t}$, $\bigcup_{e\in E(H)}P_e$ into a $t$-star-cycle using vertices from $U$ as follows. 
Let $P_1, \dots, P_{q}$ be an enumeration of $\{P_e : e \in E(H)\}$, so $q \le k \ell$. 
Let $P_{q+1} := P$.
For $j \in [q+1]$, let $x_{2j-1}, x_{2j}$ be the end vertices of~$P_j$.
For $p \in [t]$, set $x_{2q + 2p +1}, x_{2q+2p+2}$ be the center of $S'_p$ and a leaf of $S'_p$ respectively. 
By~\eqref{eqn:res2} and since $k\ell+t\leq \gamma n$, there exists distinct vertices $y_1, \dots, y_{2q+2t+2} \in U$ such that $y_j \in N_G(x_j)$ for all $j \in [2q+2t+2]$.
Moreover, \eqref{eqn:res2} implies that for each $j \in [2q+2t+2]$, $d_G ( y_j, U \cap V_i) \ge \eta |U \cap V_i|/2$ for some $i \in [k]$.
Recall that the $(\ep^{1/2}, d/2)$-reduced graph of~$G'[U]$ is connected.
By repeat applications of Lemma~\ref{lem:connect}, there exists disjoint paths $P'_1, \dots, P'_{q+t+1}$~in $G'[U]$ such that each $P'_j$ is a $(y_{2j},y_{2j+1})$-path of length at most $k+1$ (with $y_1 = y_{2q+2t+3}$).
Set
\begin{align*}
C^* = P \cup \bigcup_{p \in [t]} S^*_p \cup \bigcup_{e \in E(H)} P_e \cup \bigcup_{j \in [q+t+1]} P'_j \cup \bigcup_{j \in [2q+2t+2]}\{ x_j y_j\}. 
\end{align*}
Note that $C^*$ is a $t$-star-cycle and by~\eqref{eqn:V(C*)},
\begin{equation}\label{eqn:C*}
|V(C^*)| \ge (1 - \rho^3/2- 2\rho^8)n.
\end{equation}

From this point on, all that remains is to prove the last sentence of the Lemma.  So we assume that there exists a partition $\{A,B\}$ of~$V(G)$ such that $\delta(G[A,B]) \ge \eta n/2$.
We will show that $|A \setminus V(C^*)| = |B \setminus V(C^*)|$ by altering~$C^*$.
(Note that $W \cap V(C^*) = \emptyset$, so we can add vertices of $W$ to $C^*$.)
Let $K^*$ be the set of vertices in $C^*$ that have degree at least~$3$. 
Note that $K^*$ is precisely the set of centers of $S^*_1, \dots, S^*_t$. 
Let $A_0:=A\setminus V(C^*)$ and let $B_0:=B\setminus V(C^*)$.
Suppose without loss of generality that $|A_0|-|B_0| >0$.
Since $W \subseteq A_0 \cup B_0$, $|W\cap A| = |W \cap B|= \rho^3 n/4$, we have by \eqref{eqn:C*}
\begin{align*}
	0 < |A_0 | - |B_0| <  2\rho^8 n.
\end{align*}

First suppose that there exists $x \in K^*$ with $d_{C^*}(x) \ge \rho^4 n+2$.
Let $L^*$ be the set of vertices $y \in N_{C^*}(x)$ that have degree~$1$ in~$C^*$. 
Note that $|L^*| \ge \rho^4 n$.
If $|L^* \cap B| \ge \rho^4 n/2$, then we are done by deleting $|A_0 | - |B_0|$ vertices of $L^* \cap B$ from $C^*$. 
If $|L^* \cap A| \ge \rho^4 n/2$, then \eqref{eqn:res1} implies that 
\begin{align*}
	d_G ( x , A \cap W ) \ge \frac{\rho^3}4 |L^* \cap A| - \eps n \ge 2\rho^8 n > |A_0 | - |B_0|.
\end{align*}
In this case, we are done by joining $|A_0 | - |B_0|$ vertices in $N_G(x) \cap A \cap W$ to $x$. 

Therefore we may assume that $d_{C^*}(x) <  \rho^4 n+2\leq 2\rho^4 n$ for all $x \in K$. 
This implies that $C^*$ has at most $2s \rho^4 n$ vertices of degree~$1$. 
Let $C$ be the cycle in $C^*$, that is, $C$ is obtained from $C^*$ by deleting all vertices of degree~$1$. 
Hence by \eqref{eqn:C*},
\begin{align*}
|V(C)| \ge |V(C^*)| - 2s \rho^4 n \ge (1- \rho^3)n.
\end{align*}
Since $\delta(G[A,B]) \ge \eta n /2$, we deduce that $V(C) \cap A, V(C) \cap B  \ne \emptyset$.
Let $A' := A \setminus V(C)$ and $B' := B \setminus V(C)$ and suppose that $|A'|-|B'| >0$ (and the case $|B'|-|A'|>0$ is proved analogously).
Note that
\begin{align*}
	0 < |A'| - |B'| \le |A_0 | - |B_0| + |V(C^*) \setminus V(C)| < 	(2\rho^8+ 2s \rho^4) n \le 2(s+1) \rho^4 n.
\end{align*}
Let $x' \in V(C^*) \cap B$, which exists by \eqref{eqn:C*}. 
Recall that $\delta(G[A,B]) \ge \eta n /2$ and $\rho s \ll \eta$.
Hence \eqref{eqn:res1} implies that 
\begin{align*}
	d_G ( x' , A \cap W ) \ge \rho^3 \eta n/8 - \eps n \ge  2(s+1) \rho^4 n > |A' | - |B'|.
\end{align*}
Let $C'$ be the $1$-star-cycle obtained from~$C$ by joining $|A' | - |B'|$ vertices in $N_G(x',A \cap W)$ to~$x'$. 
Note that $|V(C')| \ge |V(C)| \ge (1- \rho^3) n$ and $|A \setminus V(C')| = |B \setminus V(C')|$, as desired.
\end{proof}

\begin{proof}[Proof of Lemma~$\ref{lma:component}$]
We will only consider the case when $G$ is $\alpha^4$-near-bipartite (as the other case can be proven by a similar argument). 
Set $\rho:=\alpha^{32/\alpha^2}$.
By Lemma~\ref{lemma:absorbing} there exists a partition $\{A,B\}$ of $V(G)$ and a path~$P$ with $|V(P)| \le \rho n $ such that $\delta(G[A,B]) \ge \eta n /2$ and for all $W\subseteq V(G)\setminus V(P)$ with $|W\cap A|=|W\cap B|\leq \rho^3 n$, the subgraph $G[V(P)\cup W]$ contains a spanning path having the same endpoints as $P$.
Now apply Lemma~\ref{lma:spanning} to $G$ to get a $t$-star cycle $C^*$ for some $t \le s$ which contains $P$ as a segment and has $|A\setminus V(C^*)|=|B\setminus V(C^*)|\leq \rho^3 n$.
By the property of $P$ and the size of the sets $A\setminus V(C^*)$ and $B\setminus V(C^*)$, we can replace~$P$ in $C$ with a path $P^*$ having the same endpoints as~$P$ and $V(P') = V(P)\cup ( V(G) \setminus V(C^*) ) $.
\end{proof}

\section{Acknowledgments}

We thank the referees for their careful reading of the paper and their suggestions for improving the exposition.

The first author would like to thank Mike Ferrara for introducing him to the problem.

\end{document}